\documentclass{amsart}
\usepackage[toc,page]{appendix}
\usepackage{geometry}
\usepackage[all,cmtip]{xy}
\usepackage{amsmath}
\usepackage{amsfonts}
\usepackage{amssymb}
\usepackage{mathrsfs}   
\usepackage{amsthm}
\usepackage{xcolor}
\usepackage{cite}
\usepackage{stmaryrd}
\usepackage{graphicx}
\usepackage{pgf}
\usepackage{eepic}
\usepackage{pstricks}
\usepackage[all,cmtip]{xy}
\usepackage{epsfig}
\usepackage{fancyhdr}
\usepackage[backref=page]{hyperref}
\renewcommand*{\backref}[1]{}
\renewcommand*{\backrefalt}[4]{({%
    \ifcase #1 Not cited.%
          \or page~#2%
          \else pages #2%
    \fi%
    })}
    
 \usepackage{mathptmx}  

\theoremstyle{plain}

\newtheorem{theorem}{Theorem}[section]

\newtheorem{proposition}[theorem]{Proposition}
\newtheorem{corollary}[theorem]{Corollary}

\newtheorem{lemma}[theorem]{Lemma}

\newtheorem*{conjectureu}{Conjecture}

\theoremstyle{definition}

\newtheorem{definition}[theorem]{Definition}

\theoremstyle{remark}
\newtheorem{remark}[theorem]{Remark}

\newcommand{\Z}{{\mathbb Z}}
\newcommand{\Q}{{\mathbb Q}}

\newcommand{\F}{{\mathbb F}}

\renewcommand{\P}{{\mathbb P}}
\newcommand{\A}{{\mathbb A}}

\newcommand{\cur}[1]{\mathcal{#1}}

\newcommand{\rig}{\mathrm{rig}}

\newcommand{\spf}[1]{\mathrm{Spf}\left(#1\right)}

\title{Incarnations of Berthelot's conjecture}

\author{Christopher Lazda}
       \address{Dipartimento di Matematica Pura e Applicata \\
        Torre Archimede, Via Trieste, 63 \\ 
        35121 Padova \\ 
        Italia}
       \email{lazda@math.unipd.it}

\begin{document}

\begin{abstract} In this article we give a survey of the various forms of Berthelot's conjecture and some of the implications between them. By proving some comparison results between push-forwards of overconvergent isocrystals and those of arithmetic $\cur{D}$-modules, we manage to deduce some cases of the conjecture from Caro's results on the stability of overcoherence under push-forward via a smooth and proper morphism of varieties. In particular, we show that Ogus' convergent push-forward of an overconvergent $F$-isocrystal under a smooth and projective morphism is overconvergent.
\end{abstract}

\maketitle 

\tableofcontents

\section{Introduction}

In many `reasonable' cohomology theories, one expects that the relative cohomology of a `fibration' $f: X\rightarrow S$ behaves as nicely as possible, that is that the higher direct image sheaves should be locally constant, and their fibres should be the cohomology groups of the fibres of $f$. For example, if one takes $f$ to be a smooth and proper morphism of algebraic varieties, then the higher direct images for de Rham cohomology (in characteristic zero) or $\ell$-adic \'{e}tale cohomology (in characteristic different from $\ell$) are `local systems' in the appropriate sense, with the expected fibres. Berthelot's conjecture is a version of this general philosophy for $p$-adic cohomology: roughly speaking, it states that if we take a smooth and proper morphism $f:X\rightarrow S$ of varieties in characteristic $p$, and an overconvergent $F$-isocrystal $E$ on $X$, then the higher direct images $\mathbf{R}^qf_*E$ should be overconvergent $F$-isocrystals on $S$. 

According to the  various different perspectives that one can take on both the coefficient objects of $p$-adic cohomology and their push-forwards, there are many different ways to state Berthelot's conjecture, some stronger, some weaker, and some (currently) logically independent, and the aim of this short article is two-fold. Firstly, it is to act as a brief survey of the various forms that Berthelot's conjecture can take, and of the special cases and impactions between them all that are currently known. Secondly, it is to show some new (but reasonably straightforward) comparisons between different constructions of push-forwards in $p$-adic cohomology, which will then allow us to deduce some new cases of certain versions of Berthelot's conjecture. While the general form of Berthelot's conjecture still remains very open, the version of it that we manage to prove here still has some interesting applications, see for example \cite{ES15b} or \cite{Pal15b}.

In the first couple of sections, therefore, we review various definitions of coefficients objects and their push-forwards, concentrating on four main perspectives: that of convergent isocrystals, overconvergent isocrystals (in two different ways) and overholonomic $\cur{D}$-modules. For each of these perspectives on coefficients objects, there is a corresponding way to phrase Berthelot's conjecture, and we are thus led to consider 4 types of conjecture. Viewing overconvergent isocrystals simply as modules with integrable connection on some frame leads to the `B' type conjectures, if we view them as modules with overconvergent stratifications, or more generally as collections of realisations with comparison morphisms, then the most natural formulation gives what we call the `S' type conjectures. If we include Frobenius structures and view them as a full subcategory of convergent $F$-isocrystals then we obtain `O' type conjectures, and finally, considering them as certain kinds of overholonomic $\cur{D}$-modules gives `C' type conjectures.

While there are reasonably clear implications between the `B', `S' and `O' type conjectures, the lack of good comparisons between push-forwards of $\cur{D}$-modules and push-forwards of overconvergent isocrystals in general means that there are few straightforward implications between the `C' conjectures and the others. Since it is the `C' conjectures for which, thanks to Caro's work, most cases are known (in particular, all quasi-projective cases) it is therefore especially disappointing that it is these `C' conjectures that are the most difficult to relate to the others. Our rather modest contribution here is to note a few special cases of such comparison theorems between push-forwards, which enables us to deduce `O' type conjectures with a reasonably respectable level of generality (namely, for smooth \emph{projective} morphisms $X\rightarrow S$), and `B' type conjectures with a somewhat less respectable level of generality (see Corollary \ref{bp} for a precise statement).

The main difficulty in extending these results is the somewhat indirect comparison between overconvergent isocrystals and overcoherent isocrystals (which are certain kinds of arithmetic $\cur{D}$-modules). The equivalence of categories constructed by Caro makes fundamental use of both resolution and gluing arguments, and therefore if one is to obtain the required comparisons between push-forwards, one needs to know certain cases of finiteness and base change for rigid higher direct images in order to push these objects through the construction - in other words, one needs to know certain cases of `S' or `B' type conjectures before one starts! The reasons that we could get our arguments to work here is essentially by bootstrapping up the few cases in which one has a direct comparison between overconvergent and overcoherent isocrystals as far as possible, which in `O' type conjectures does in fact give reasonable results, but is still rather inadequate for `B' or `S' type conjectures. One would hope that a direct comparison would lead to an easy implication from the `C' type conjectures proved by Caro to the conjectures of the other types.

\subsection*{Notations and conventions}

Throughout, $k$ will be a perfect field of characteristic $p>0$, $\cur{V}$ will be a complete DVR of mixed characteristic with residue field $k$ and fraction field $K$, and $\pi$ will be a uniformiser for $\cur{V}$. A $k$-variety will mean a separated $k$-scheme of finite type, and a formal $\cur{V}$-scheme will mean a $\pi$-adic formal scheme, separated and of finite type over $\spf{\cur{V}}$. If $X$ is an $\F_p$ scheme, absolute Frobenius will mean some fixed power of the $p$-power absolute Frobenius on $X$. For any $k$-variety $X$ we will denote the reduced subscheme by $X_\mathrm{red}$. For any formal $\cur{V}$-scheme $\frak{X}$, we will denote the special fibre by $\frak{X}_0$, and the generic fibre by $\frak{X}_K$, this is a rigid space over $K$ in the sense of Tate. If $\cur{F}$ is an abelian sheaf on some site, we will denote by $\cur{F}_\Q$ the tensor product $\cur{F}\otimes_{\Z}\Q$.

\section{Categories of isocrystals}\label{coeffs}

In this section, we review the various categories of coefficients that are used in $p$-adic cohomology, and the various comparison theorems between them. We start with the category of convergent isocrystals, following Ogus \cite{Ogu84}.

Let $X$ be a $k$-scheme. The convergent site of $X/\cur{V}$ consists of pairs $(\frak{T},z_\frak{T})$ where $\frak{T}$ is a flat formal $\cur{V}$-scheme and $z_\frak{T}:(\frak{T}_0)_\mathrm{red}\rightarrow X$ is a morphism of $k$-varieties. The topology is induced by the Zariski topology on $\frak{T}$, and the associated topos is denoted $(X/\cur{V})_\mathrm{conv}$. We will usually drop $z_\frak{T}$ from the notation, and refer to an object of the convergent site simply as $\frak{T}$. We can describe sheaves $E$ on this site as `realisations' $E_\frak{T}$ and transition morphisms 
$$ g^{-1}E_\frak{T}\rightarrow E_{\frak{T}'}
$$
associated to $g:\frak{T}'\rightarrow \frak{T}$ in the usual way. In particular we have the canonical sheaf $\cur{K}_{X/V}$ whose realisation on $\frak{T}$ is $\cur{O}_{\frak{T},\Q}$.

\begin{definition} A convergent isocrystal on $X$ is a $\cur{K}_{X/\cur{V}}$-module $E$ such that each realisation $E_\frak{T}$ is a coherent $\cur{O}_{\frak{T},\Q}$-module, and the \emph{linearised} transition morphism
$$ g^*E_\frak{T}\rightarrow E_{\frak{T}'}
$$
associated to any $g:\frak{T}'\rightarrow \frak{T}$ is an isomorphism. The category of such objects is denoted $\mathrm{Isoc}(X/K)$. 
\end{definition}

\begin{proposition}[\cite{Ogu84}, Theorem 2.15] \label{convcon} Suppose that $\frak{X}$ is a smooth formal $\cur{V}$-scheme with special fibre $X$. Then the realisation functor $E\mapsto E_\frak{X}$ induces a fully faithful functor from $\mathrm{Isoc}(X/K)$ to the category of coherent $\cur{O}_{\frak{X},\Q}$-modules with integrable connection.
\end{proposition}

\begin{remark} There is also a version of this proposition where we embed $X$ into a smooth formal $\cur{V}$-scheme. We will see this appearing later on.
\end{remark}

These objects are functorial in both $X$ and $\cur{V}$, in that a commutative diagram
$$ \xymatrix{ Y \ar[r] \ar[d] & X \ar[d] \\
\spf{\cur{W}} \ar[r] & \spf{\cur{V}} 
}
$$
induces a pullback functor $\mathrm{Isoc}(X/K)\rightarrow \mathrm{Isoc}(Y/L)$ (where $L=\mathrm{Frac}(\cur{W})$). In particular, after choosing a lift to $\cur{V}$ of the absolute Frobenius of $k$, we can talk about convergent isocrystals with Frobenius structure, the category of such objects being denoted $F\textrm{-}\mathrm{Isoc}(X/K)$, and there is an analogue of Proposition \ref{convcon} if $\frak{X}$ is equipped with a lift of absolute Frobenius.

Next we introduce (partially) overconvergent isocrystals, following Berthelot \cite{Ber96b} and Le Stum \cite{LS07}. Before we do so we need to introduce pairs and frames, as well as Berthelot's functor $j^\dagger$ of overconvergent sections.

\begin{definition} A $k$-pair consists of an open embedding $X\rightarrow \overline{X}$ of $k$-varieties. A $\cur{V}$-frame consists of a $k$-pair $(X,\overline{X})$ and a closed embedding $\overline{X}\rightarrow \frak{X}$ of $\overline{X}$ into a formal $\cur{V}$-scheme. We will say that a pair/frame is proper if $\overline{X}$ is, and that a frame is smooth if $\frak{X}$ is smooth in a neighbourhood of $X$. A morphism of pairs/frames is just a commutative diagram, and we will say a morphism of pairs $(X,\overline{X})\rightarrow (S,\overline{S})$ is Cartesian if the associated commutative square is. Smoothness/properness of a morphism of pairs or frames is defined as before. If $(X,\overline{X})$ is a pair, then a frame over $(X,\overline{X})$ is a frame $(Y,\overline{Y},\frak{Y})$ together with a morphism $(Y,\overline{Y})\rightarrow (X,\overline{X})$ of pairs.
\end{definition}

If we have a frame $(X,\overline{X},\frak{X})$, then we can consider the specialisation map $\mathrm{sp}:\frak{X}_K\rightarrow \frak{X}_0$, and for any locally closed subscheme $V\subset \frak{X}_0$ we define the tube
$$ ]V[_\frak{X}:= \mathrm{sp}^{-1}(V).
$$
If $]X[_\frak{X}\subset V \subset ]\overline{X}[_\frak{X}$ is an open subset of $]\overline{X}[_\frak{X}$, then we will call $V$ a strict neighbourhood of $]X[_\frak{X}$ if the covering
$$ ]\overline{X}[_\frak{X} = V \cup ]\overline{X}\setminus X[_\frak{X}
$$
is admissible for the $G$-topology. For any sheaf $\cur{F}$ on $]\overline{X}[_\frak{X}$ we define
$$ j_X^\dagger\cur{F}:=\mathrm{colim}_V j_{V*}j_V^{-1}\cur{F}
$$
where the colimit is taken over all strict neighbourhoods $V$, and $j_V:V\rightarrow ]\overline{X}[_\frak{X}$ denotes the inclusion. If $E$ is a $j^\dagger_X\cur{O}_{]\overline{X}[_\frak{X}}$-module, then an integrable connection on $E$ is just an integrable connection on $E$ as an $\cur{O}_{]\overline{X}[_\frak{X}}$-module. 

\begin{definition} An overconvergent isocrystal on the pair $(X,\overline{X})$ consists of a collection $E_\frak{U}$ of coherent $j^\dagger_Y\cur{O}_{]\overline{Y}[_\frak{Y}}$-modules, one for each frame $(Y,\overline{Y},\frak{Y})$ over $(X,\overline{X})$, together with isomorphisms $u^*E_\frak{Y}\rightarrow E_\frak{Z}$ associated to each morphism $u:(Z,\overline{Z},\frak{Z})\rightarrow (Y,\overline{Y},\frak{Y})$ of frames, satisfying the usual cocycle conditions. The category of such objects is denoted $\mathrm{Isoc}^\dagger((X,\overline{X})/K)$, and we refer to the $E_\frak{Y}$ as the realisations of $E$.
\end{definition}

\begin{proposition}[\cite{LS07}, Proposition 7.2.13] Suppose that $(X,\overline{X},\frak{X})$ is a smooth frame. Then the realisation functor $E\mapsto E_\frak{X}$ induces a fully faithful functor from $\mathrm{Isoc}^\dagger((X,\overline{X})/K)$ to the category $\mathrm{MIC}((X,\overline{X},\frak{X})/K)$ of coherent $j^\dagger\cur{O}_{]\overline{X}[_\frak{X}}$-modules with integrable connection.
\end{proposition}

Of course, as before, we have functoriality as well as a version with Frobenius structures, denoted $F\textrm{-}\mathrm{Isoc}^\dagger((X,\overline{X})/K)$. The category $(F\textrm{-})\mathrm{Isoc}(X/K)$ is local on $X$, and $(F\textrm{-})\mathrm{Isoc}^\dagger((X,\overline{X})/K)$ is local on both $X$ and $\overline{X}$. When the pair $(X,\overline{X})$ is proper, $(F\textrm{-})\mathrm{Isoc}^\dagger((X,\overline{X})/K)$ depends only on $X$, we will therefore write $(F\textrm{-})\mathrm{Isoc}^\dagger(X/K)$.

\begin{definition} Let $(X,\overline{X},\frak{X})$ be a smooth frame. Then an overconvergent stratification on a coherent $j_X^\dagger\cur{O}_{]\overline{X}[_\frak{X}}$-module $E$ is an isomorphism
$$ p_2^*E\cong p_1^*E
$$
of $j_X^\dagger\cur{O}_{]\overline{X}[_{\frak{X}^2}}$-modules, where $\overline{X}$ is embedded in $\frak{X}^2$ via the diagonal, and $p_i:\frak{X}^2\rightarrow \frak{X}$ are the two projections. This isomorphism is subject to the usual conditions, for example it should be the identity after being pulled back via $\Delta:\frak{X}\rightarrow \frak{X}^2$, and should satisfy a cocycle condition on $\frak{X}^3$. The category of coherent $j_X^\dagger\cur{O}_{]\overline{X}[_\frak{X}}$-modules with overconvergent stratification is denoted by $\mathrm{Strat}^\dagger((X,\overline{X},\frak{X})/K)$. There is an obvious restriction functor $$\mathrm{Isoc}^\dagger((X,\overline{X})/K)\rightarrow \mathrm{Strat}^\dagger((X,\overline{X},\frak{X})/K)$$ which is an equivalence by Proposition 7.2.2 of \cite{LS07}. 
\end{definition}

By restricting to frames of the form $(\frak{Y}_0,\frak{Y}_0,\frak{Y})$ we also get a natural functor
$$ (F\textrm{-})\mathrm{Isoc}^\dagger((X,X)/K) \rightarrow (F\textrm{-})\mathrm{Isoc}(X/K).
$$

\begin{proposition}[\cite{Ber96b}, 2.3.4] This functor is an equivalence of categories. \end{proposition}

\begin{remark} This give an answer as to what the analogue of Proposition \ref{convcon} should be when $X$ is not smooth over $k$: convergent isocrystals form a full subcategory of the category of coherent $\cur{O}_{]X[_\frak{X}}$-modules with integrable connection, for any closed embedding $X\rightarrow \frak{X}$ into a smooth formal $\cur{V}$-scheme. In fact, it is this characterisation which is the key ingredient in the proof of the previous proposition.
\end{remark}

For any pair $(X,\overline{X})$ we get a canonical restriction functor 
$$ (F\textrm{-})\mathrm{Isoc}^\dagger((X,\overline{X})/K) \rightarrow (F\textrm{-})\mathrm{Isoc}(X/K)
$$
and have the following theorem of Caro and Kedlaya.

\begin{theorem}[\cite{Car11}, Th\'{e}or\`{e}me 2.2.1] The restriction functor
$$  F\textrm{-}\mathrm{Isoc}^\dagger((X,\overline{X})/K)\rightarrow F\textrm{-}\mathrm{Isoc}(X/K)
$$
is fully faithful.
\end{theorem}

\begin{remark}This is also conjectured to hold without Frobenius structures, but this is not currently known.
\end{remark}

Finally, the most complicated category of coefficients is that of Berthelot's arithmetic $\cur{D}$-modules, as developed by Caro. Since the definitions and constructions are so involved, we will content ourselves with giving a brief overview, referring to Berthelot and Caro's work for the details.

If $\frak{P}$ is a smooth formal $\cur{V}$-scheme, then we let $\cur{D}^\dagger_{\frak{P},\Q}$ denote the ring of overconvergent differential operators on $\frak{P}$, as constructed in \S2 of \cite{Ber96a}, and $D^b_\mathrm{coh}(\cur{D}^\dagger_{\frak{P},\Q})$ its (bounded, coherent) derived category. For any closed subscheme $T\subset \frak{P}_0$, we have functors $\mathbf{R}\underline{\Gamma}^\dagger_T$ and $(^\dagger T)$ of `sections with support in $T$' and `sections overconvergent along $T$' respectively, and an exact triangle
$$ \mathbf{R}\underline{\Gamma}_T^\dagger\cur{E} \rightarrow \cur{E} \rightarrow (^\dagger T)\cur{E} \rightarrow \mathbf{R}\underline{\Gamma}_T^\dagger\cur{E}[1]
$$
for any $\cur{E}\in D^b_\mathrm{coh}(\cur{D}^\dagger_{\frak{P},\Q})$. We let $D^b_\mathrm{surcoh}(\cur{D}^\dagger_{\frak{P},\Q})\subset D^b_\mathrm{coh}(\cur{D}^\dagger_{\frak{P},\Q})$ denote the full subcategory of overcoherent objects, as defined in \S3 of \cite{Car04}. 

We will also need a variant: if $T\subset \frak{P}_0$ is a divisor of the special fibre of $\frak{P}$, we may consider the ring $\cur{D}^\dagger_{\frak{P}}(^\dagger T)_\Q$ of differential operators with overconvergent singularities along $T$, as defined in \S4 of \cite{Ber96a}, as well as the categories $D^b_\mathrm{coh}(\cur{D}^\dagger_{\frak{P}}(^\dagger T)_\Q)$ and $D^b_\mathrm{surcoh}(\cur{D}^\dagger_{\frak{P}}(^\dagger T)_\Q)$ as before. There is a forgetful functor 
$$ D^b_\mathrm{coh}(\cur{D}^\dagger_{\frak{P}}(^\dagger T)_\Q) \rightarrow D^b_\mathrm{coh}(\cur{D}^\dagger_{\frak{P},\Q})
$$
which is fully faithful and with essential image those objects $\cur{E}$ such that $\cur{E}\cong(^\dagger T)\cur{E}$ (Lemme 1.2.1 4 of \cite{Car15a}). Be warned, however, that it does not respect the notion of overcoherence in general.

Now let $(X,\overline{X})$ be a $k$-pair, and assume that we have an embedding $\overline{X}\hookrightarrow \tilde{\frak{P}}$ into a smooth and proper formal $\cur{V}$-scheme, and a divisor $\tilde{T}\subset \tilde{\frak{P}}_0$ such that $X=\overline{X}\setminus \tilde{T}$. Let $\frak{P}$ be an open formal subscheme of $\tilde{\frak{P}}$ such that $\overline{X}\rightarrow \tilde{\frak{P}}$ factors through a closed embedding $\overline{X}\rightarrow \frak{P}$, and let $T=\tilde{T}\cap \frak{P}$. Then the full subcategory of $D^b_\mathrm{surcoh}(\cur{D}^\dagger_{\frak{P}}(^\dagger T))$ consisting of objects with support in $\overline{X}$, i.e. such that $\cur{E}\cong \mathbf{R}\underline{\Gamma}^\dagger_{\overline{X}}\cur{E}$, is independent of all choices (i.e. only depends on $(X,\overline{X})$), we therefore denote it $D^b_\mathrm{\mathrm{surcoh}}(\cur{D}^\dagger_{(X,\overline{X})/K})$ and refer to it as the category of overcoherent $\cur{D}^\dagger$-modules on $(X,\overline{X})/K$. When $X=\overline{X}$ we will denote it instead by  $D^b_\mathrm{\mathrm{surcoh}}(\cur{D}_{X/K})$. When $\overline{X}$ is proper, it depends only on $X$ and we will therefore instead write $D^b_\mathrm{surcoh}(\cur{D}^\dagger_{X/K})$.

\begin{remark} We formalise the hypothesis used on pairs in the previous paragraph as follows: A couple $(Y,\overline{Y})$ is said to be `properly $d$-realisable' if there exists smooth and proper formal $\cur{V}$-scheme $\frak{P}$, a (not necessarily closed) immersion $\overline{Y}\rightarrow \frak{P}$ and a divisor $D$ of $\frak{P}_0$ such that $Y=\overline{Y}\setminus D$.
\end{remark}

\begin{lemma} \label{annoying} Let $\frak{S}$ be a smooth affine formal $\cur{V}$-scheme, with special fibre $S$. Then there exists a canonical equivalence of categories
$$ D^b_{\mathrm{surcoh}}(\cur{D}^\dagger_{\frak{S,\Q}}) \cong D^b_{\mathrm{surcoh}}(\cur{D}_{S/K}).
$$
More generally, if $D\subset \overline{S}:=\frak{S}_0$ is a divisor, and $S=\overline{S}\setminus D$, then there exists a canonical equivalence of categories
$$ D^b_{\mathrm{surcoh}}(\cur{D}^\dagger_{\frak{S}}(^\dagger D)_\Q) \cong D^b_{\mathrm{surcoh}}(\cur{D}^\dagger_{(S,\overline{S})/K}).
$$
\end{lemma}

\begin{proof} Note that this is not immediate from the definitions! The first is not difficult: if we choose an affine embedding $\frak{S}\hookrightarrow \widehat{\A}^n_\cur{V}$ then $D^b_{\mathrm{surcoh}}(\cur{D}_{S/K})$ can be identified with the full subcategory of $D^b_{\mathrm{surcoh}}(\cur{D}^\dagger_{\frak{S,\Q}})$ consisting of objects with support in $\frak{S}$, the claimed equivalence therefore follows from the Berthelot-Kashiwara theorem (Th\'{e}or\`{e}me 3.1.6 of \cite{Car04}). The second is proved entirely similarly.
\end{proof}

Whenever $X$ is smooth, we will let $\mathrm{Isoc}^{\dagger\dagger}((X,\overline{X})/K) \subset D^b_\mathrm{surcoh}(\cur{D}^\dagger_{(X,\overline{X})/K})$ denote the full subcategory of `overcoherent isocrystals' as in D\'{e}finition 1.2.4 of \cite{Car15a}, these are certain kinds of overcoherent $\cur{D}$-modules, and we will denote by $D^b_\mathrm{isoc}(\cur{D}^\dagger_{(X,\overline{X})/K})$ the full subcategory of $D^b_\mathrm{surcoh}(\cur{D}^\dagger_{(X,\overline{X})/K})$ objects whose cohomology sheaves are overcoherent isocrystals. We have a canonical equivalence of categories
$$ \mathrm{sp}_{(X,\overline{X}),+} :\mathrm{Isoc}^\dagger((X,\overline{X})/K) \rightarrow \mathrm{Isoc}^{\dagger\dagger}((X,\overline{X})/K)
$$
whose description we will need in the following two special cases.

\begin{itemize} \item Assume that $\frak{X}$ is a smooth formal $\cur{V}$-scheme, so that we may identify objects of $\mathrm{Isoc}(X/K)$ with certain $\cur{O}_{\frak{X},\Q}$-modules with integrable connection, and objects of $D^b_\mathrm{surcoh}(\cur{D}_{X/K})$ with a full subcategory of $D^b_\mathrm{coh}(\cur{D}^\dagger_{\frak{X},\Q})$. Then $\mathrm{sp}_{(X,X),+}$ simply takes a module with integrable connection to the corresponding $\cur{D}$-module, as in Proposition 4.1.4 of \cite{Ber96a}.  \label{sp1}
\item Assume that $\frak{X}$ is a smooth formal $\cur{V}$-scheme, that $T\subset \overline{X}:=\frak{X}_0$ is a divisor, and set $X=\overline{X}\setminus T$. Let $\mathrm{sp}_*:\frak{X}_K\rightarrow\frak{X}$ be the specialisation map. Then thanks to Proposition 4.4.2 of \cite{Ber96a}, $\mathrm{sp}_*$ induces an equivalence of categories between $\mathrm{Isoc}^\dagger((X,\overline{X})/K)$ and certain coherent $\cur{O}_{\frak{X},\Q}(^\dagger T)$-modules with integrable connection. Since we may identify objects of $D^b_\mathrm{surcoh}(\cur{D}^\dagger_{(X,\overline{X})/K})$ with a full subcategory of $D^b_\mathrm{coh}(\cur{D}^\dagger_{\frak{X}}(^\dagger T))$, the functor $\mathrm{sp}_{(X,\overline{X}),+}$ is again that taking an integrable connection to the associated $\cur{D}$-module, as in Th\'{e}or\`{e}me 4.4.5 of \emph{loc. cit}. \label{sp2}
\end{itemize}

\begin{remark} One can also construct  $D^b_\mathrm{isoc}(\cur{D}^\dagger_{(X,\overline{X})})$ for non-smooth $X$, although the definition is more involved, and it is not clear whether or not we have the equivalence of categories 
$$ \mathrm{Isoc}^\dagger((X,\overline{X})/K) \cong \mathrm{Isoc}^{\dagger\dagger}((X,\overline{X})/K)
$$
in this case.
\end{remark}

To define $D^b_\mathrm{\mathrm{surcoh}}(\cur{D}^\dagger_{(X,\overline{X})/K})$ and $D^b_\mathrm{isoc}(\cur{D}^\dagger_{(X,\overline{X})/K})$ in general we use Zariski descent: by localising on $\overline{X}$ and $X$ we may assume that we are in the `properly $d$-realisable' situation as above, and the corresponding categories glue. For more details on how to do this, see for example Remarque 3.2.10 of \cite{Car04}.

\section{Push-forwards in $p$-adic cohomology}

For the various different categories of $p$-adic coefficients, there are different ways of viewing higher direct images, and in this section we will review the basic constructions
 of such push-forwards. Again, we start with convergent isocrystals. 

Suppose that $X$ is a $k$-variety, $\frak{S}$ is a formal $\cur{V}$-scheme, and $f:X\rightarrow \frak{S}$ is a morphism of formal $\cur{V}$-schemes. Then we can define the category of convergent isocrystals on $X/\frak{S}$ exactly as in \S\ref{coeffs}, only taking formal schemes with a given structure morphism to $\frak{S}$. There is a canonical morphism of topoi
$$ (X/\cur{V})_\mathrm{conv} \rightarrow (X/\frak{S})_\mathrm{conv}
$$
induced by `forgetting' the structure morphism to $\frak{S}$; push-forward is exact and sends isocrystals to isocrystals. We then consider the morphism of topoi
$$ f_{\frak{S},\mathrm{conv}}:(X/\frak{S})_\mathrm{conv} \rightarrow \frak{S}_\mathrm{Zar}
$$
induced by the functor taking an open subset of $\frak{S}$ to the object $(f^{-1}\frak{U},\frak{U})$  of $(X/\frak{S})_\mathrm{conv}$. For any convergent isocrystal on $X/\frak{S}$, we can therefore consider the $\cur{O}_{\frak{S},\Q}$-modules
$$  \mathbf{R}^qf_{\frak{S},\mathrm{conv}*}E\in \frak{S}_\mathrm{Zar}.
$$

\begin{proposition} \label{derham} Suppose that $f:\frak{X}\rightarrow\frak{S}$ is a smooth morphism of formal $
\cur{V}$-schemes lifting $f:X\rightarrow \frak{S}$. 
\begin{enumerate} 
\item There is an equivalence of categories $E\mapsto E_\frak{X}$ between convergent isocrystals on $X/\frak{S}$ and a full subcategory of the category of coherent $\cur{O}_{\frak{X},\Q}$-modules with integrable connection relative to $\frak{S}$.
\item For any convergent isocrystal $E$ on $X/\frak{S}$, there is a canonical isomorphism
$$ \mathbf{R}^qf_{\frak{S},\mathrm{conv}*}E \cong \mathbf{R}^qf_{*} ( E_\frak{X} \otimes\Omega^*_{\frak{X}/\frak{S}}).
$$
\end{enumerate}
\end{proposition}

\begin{proof} Note that Proposition 2.21 of \cite{Shi08a} (in the exceptionally simple case where the log structure is trivial and $\cur{P}_0=X$) implies that there is an equivalence of categories between convergent isocrystals and coherent $\cur{O}_{\frak{X}_K}$ modules with a convergent integrable connection, which implies (1), (2) then follows immediately from Corollary 2.33 of \emph{loc. cit}. \end{proof}

Hence, by localising, whenever $X$ is smooth over $\frak{S}_0$, $\frak{S}$ is smooth, and $E\in \mathrm{Isoc}(X/K)$, the $\cur{O}_{\frak{S},\Q}$-modules $\mathbf{R}^qf_{\frak{S},\mathrm{conv}*}E$ are equipped with a canonical connection, the Gauss--Manin connection. In fact, the assumption that $X$ is smooth over $\frak{S}_0$ is unnecessary, but we will not need this directly . Also, if $\frak{S}$ comes equipped with a lift $\sigma_\frak{S}$ of the absolute Frobenius of $\frak{S}_0$, and $E\in F\textrm{-}\mathrm{Isoc}^\dagger(X/K)$, then the sheaf $\mathbf{R}^qf_{\frak{S},\mathrm{conv}*}E$ has a natural Frobenius morphism
$$ \sigma_\frak{S}^* \mathbf{R}^qf_{\frak{S},\mathrm{conv}*}E \rightarrow \mathbf{R}^qf_{\frak{S},\mathrm{conv}*}E
$$
which is compatible with the connection. 

Now assume that we have a smooth and proper morphism $f:X\rightarrow S$ of $k$-varieties. Then in \S3 of \cite{Ogu84}, Ogus constructs, for any convergent $F$-isocrystal $E\in F\textrm{-}\mathrm{Isoc}(X/K)$, and $q\geq0$, a convergent isocrystal $\mathbf{R}^qf_{\mathrm{conv}*}E\in F\textrm{-}\mathrm{Isoc}(X/K)$ using crystalline cohomology (actually he only does this for the constant $F$-isocrystal, but essentially the same construction works in general, using Th\'{e}or\`{e}me 2.4.2 of \cite{Ber96b}). The construction is compatible with base change, in that if we have a Cartesian diagram
$$ \xymatrix{ X'\ar[r]^{g'}\ar[d]_{f'} & X\ar[d]^f \\ S'\ar[r]^g & S 
}
$$
then there is a natural isomorphism
$$ g^*\mathbf{R}^qf_{\mathrm{conv}*}E\cong \mathbf{R}^qf'_{\mathrm{conv}*}g'^*E
$$
in $F\textrm{-}\mathrm{Isoc}(S'/K)$.

\begin{proposition} \label{pfconv} Suppose that $\frak{S}$ is smooth over $\cur{V}$, and that $\sigma_\frak{S}$ is a lift of Frobenius. Then for any smooth and proper morphism $f:X\rightarrow\frak{S}_0$ of $k$-varieties, the realisation $(\mathbf{R}^qf_{\mathrm{conv}*}E)_\frak{S}$ is isomorphic to $\mathbf{R}^qf_{\frak{S},\mathrm{conv}*}E$ with its canonical Frobenius and connection. 
\end{proposition}

\begin{proof} If we ignore Frobenius structures, then this just follows from the fact that $\mathbf{R}^qf_{\mathrm{conv}*}E$ is constructed using crystalline cohomology, and hence the induced connection is simply the Gauss--Manin connection. Compatibility with Frobenius then follows from compatibility with base change in general.
\end{proof}

Next we turn to push-forward for overconvergent isocrystals.

\begin{definition} Let $f:(X,\overline{X},\frak{X})\rightarrow (S,\overline{S},\frak{S})$ be a smooth morphism of smooth frames, and $E\in (F)\textrm{-}\mathrm{Isoc}^\dagger((X,\overline{X})/K)$. Then define
$$ \mathbf{R}^qf_{\frak{S},\rig*}E := \mathbf{R}f_*(E_\frak{X}\otimes \Omega^*_{]\overline{X}[_\frak{X}/]\overline{S}[_\frak{S}}).
$$
\end{definition}

This only depends on the induced morphism $f:(X,\overline{X})\rightarrow (S,\overline{S},\frak{S})$ and not on the choice of $\frak{X}$, and if $\overline{X}$ is proper, then it in fact only depends on $f:X\rightarrow (S,\overline{S},\frak{S})$. The construction is local on $\overline{X}$, and hence to define $\mathbf{R}^qf_{\frak{S},\rig*}E$ for an arbitrary pair $(X,\overline{X})$ over $(S,\overline{S},\frak{S})$ (i.e. one not necessarily admitting an extension to a smooth morphism of frames $(X,\overline{X},\frak{X})\rightarrow (S,\overline{S},\frak{S})$) we can use descent. The details are somewhat tedious, so we won't go into them here. For a detailed description of how this works, see \cite{CT03}.

Note that $\mathbf{R}^qf_{\frak{S},\rig*}E$ is a $j^\dagger_S\cur{O}_{]\overline{S}[_\frak{S}}$-module, and is equipped with a canonical connection, the Gauss--Manin connection. When $\frak{S}$ admits a lift of Frobenius, then these sheaves are also equipped with a Frobenius morphism. Using the restriction functor 
$$  (F\textrm{-})\mathrm{Isoc}^\dagger((X,\overline{X})/K)\rightarrow (F\textrm{-})\mathrm{Isoc}(X/K)
$$
we can also define $\mathbf{R}^qf_{\frak{S},\mathrm{conv}*}E$ for any $E\in (F\textrm{-})\mathrm{Isoc}^\dagger((X,\overline{X})/K)$, and $\mathbf{R}^qf_{\mathrm{conv}*}E$ whenever $X\rightarrow S$ is smooth and proper and $E\in F\textrm{-}\mathrm{Isoc}^\dagger((X,\overline{X})/K)$. The relation between these is as follows.

\begin{proposition}\label{occ} Let $(S,\overline{S},\frak{S})$ be a smooth frame, and suppose that $f:(X,\overline{X})\rightarrow (S,\overline{S})$ is a Cartesian morphism of pairs. Assume that $\frak{S}$ admits a lift of the absolute Frobenius of $\overline{S}$, and let $\frak{S}'$ be an open subset of $\frak{S}$, stable under Frobenius, such that $\frak{S}'\cap \overline{S}=S$. Then for any $E\in F\textrm{-}\mathrm{Isoc}^\dagger((X,\overline{X})/K)$ the restriction of $\mathbf{R}^qf_{\frak{S},\rig*}E$ to $]S[_\frak{S} = ]S[_{\frak{S'}}$ is isomorphic to the realisation of $\mathbf{R}^qf_{\mathrm{conv}*}E\in F\textrm{-}\mathrm{Isoc}(S/K)\cong F\textrm{-}\mathrm{Isoc}^\dagger((S,S)/K)$ on $(S,S,\frak{S}')$. 
\end{proposition}

\begin{proof} We may assume that $S=\overline{S}$, and $\frak{S}'=\frak{S}$, where the claim follows from Corollary 2.34 of \cite{Shi08a}. (Shiho actually treats the more general case of log schemes, which includes the above as a special case). 
\end{proof}

\begin{remark} Shiho's relative log convergent cohomology is used in \cite{CT14} to obtain a Clemens-Schmidt type exact sequence in $p$-adic cohomology.
\end{remark}

Again, the most involved of the push-forward constructions is the version in Berthelot's theory of arithmetic $\cur{D}$-modules, and we will only give the briefest of overviews here. So suppose that $f:\frak{P}'\rightarrow \frak{P}$ is a smooth morphism of formal $\cur{V}$-schemes. Then Berthelot constructs in (4.3.7) of \cite{Ber02} a $(f^{-1}\cur{D}^\dagger_{\frak{P},\Q},\cur{D}^\dagger_{\frak{P}',\Q})$-bimodule $\cur{D}^\dagger_{\frak{P}'\rightarrow \frak{P},\Q}$ and defines the push-forward of a complex $\cur{E}\in D^b_\mathrm{coh}(\cur{D}^\dagger_{\frak{P}',\Q})$ to be
$$ f_+ \cur{E} := \mathbf{R}f_*(\cur{D}^\dagger_{\frak{P}'\rightarrow \frak{P},\Q}\otimes^\mathbf{L}_{\cur{D}^\dagger_{\frak{P}',\Q}} \cur{E}).
$$
There is a similar version for differential operators overconvergent along a divisor, assuming that the pullback of the divisor on $\frak{P}$ is contained in that on $\frak{P}'$.

If $f:(X,\overline{X})\rightarrow (S,\overline{S})$ is a \emph{proper} morphism of properly $d$-realisable pairs, then we may construct a diagram
$$ \xymatrix{ \overline{X} \ar[r]\ar[d]_f & \frak{P}'\ar[r]\ar[d] & \tilde{\frak{P}}' \ar[d]^g \\
\overline{S} \ar[r] & \frak{P} \ar[r] & \tilde{\frak{P}} 
}
$$
with both left had horizontal arrows closed immersions, both right hand horizontal arrows open immersions, $\tilde{\frak{P}}'\rightarrow\tilde{\frak{P}}$ a smooth and proper morphism between smooth and proper formal $\cur{V}$ schemes, and the right hand square Cartesian, such that there exist divisors $T,D$ of $\tilde{\frak{P}}_0',\tilde{\frak{P}}_0$ respectively with $X=\overline{X}\setminus T$, $S=\overline{S}\setminus D$, and $g^{-1}(D)\subset T$, as in Lemme 4.2.8 of \cite{Car15a}. We may then define the push-forward
$$ f_+:D^b_\mathrm{surcoh}(\cur{D}^\dagger_{(X,\overline{X})/K}) \rightarrow D^b_\mathrm{surcoh}(\cur{D}^\dagger_{(S,\overline{S})/K})
$$
as simply the push-forward $g_+$ associated to the lift $g:\frak{P}'\rightarrow \frak{P}$: this does indeed land in the category $D^b_\mathrm{surcoh}(\cur{D}^\dagger_{(S,\overline{S})/K})\subset D^b_\mathrm{surcoh}(\cur{D}^\dagger_{\frak{S}}(^\dagger D)_\Q)$, and does not depend on any of the choices (see Proposition 4.2.7 of \emph{loc. cit.}). 

\section{Versions of Berthelot's conjecture}\label{conjs}

According to the different interpretations of the category of overconvergent ($F$)-isocrystals, there are correspondingly different versions of Berthelot's conjecture, each of which is most naturally adapted to a particular viewpoint on the category of isocrystals. In this section, we review some of the different versions of Berthelot's conjecture, and discuss some of the easier implications among them. We start with Berthelot's original formulation, which is the one most closely related to the viewpoint of overconvergent isocrystals as $j^\dagger\cur{O}$-modules with connection.

\begin{conjectureu}[B(F), \cite{Ber86} \S4.3] Let $(S,\overline{S},\frak{S})$ be a smooth and proper $\cur{V}$-frame, and $f:X\rightarrow S$ a smooth and proper morphism of $k$-varieties. Then the $j-S^\dagger\cur{O}_{]\overline{S}[_\frak{S}}$-module with integrable connection $\mathbf{R}^qf_{\frak{S},\rig*}\cur{O}_{X/K}^\dagger$ 
arises from a unique overconvergent ($F$)-isocrystal
$$ \mathbf{R}^qf_{\rig*}\cur{O}_{X/K}^\dagger \in (F)	\textrm{-}\mathrm{Isoc}^\dagger(S/K).
$$
In other words, $\mathbf{R}^qf_{\rig*}\cur{O}_{X/K}^\dagger$ is coherent, the connection is overconvergent, and the resulting object in $\mathrm{Isoc}^\dagger(S/K) \cong \mathrm{MIC}^\dagger((S,\overline{S},\frak{S})/K)$ only depends on $S$ and not on the choice of frame $(S,\overline{S},\frak{S})$. (Moreover, this object has a canonical Frobenius structure.)
\end{conjectureu}

\begin{remark} When referring to this and any other form of Berthelot's conjecture, we will use, for example, `Conjecture B' to refer to the conjecture without Frobenius structure, and `Conjecture  BF' to refer to the conjecture with Frobenius structure.
\end{remark}

We also have the following slightly more general formulation of Berthelot's original conjecture, due to Tsuzuki.

\begin{conjectureu}[B1(F), \cite{Tsu03}] Suppose that $(X,\overline{X})\rightarrow (S,\overline{S})$ is a proper, Cartesian morphism of $k$-pairs with $X\rightarrow S$ smooth, and that $(S,\overline{S},\frak{S})$ is a smooth $\cur{V}$-frame. Let $E\in (F)\textrm{-}\mathrm{Isoc}^\dagger((X,\overline{X})/K)$. Then the $j_S^\dagger\cur{O}_{]\overline{S}[_\frak{S}}$-module with integrable connection $\mathbf{R}^qf_{\frak{S},\rig*}E$ 
arises from a unique overconvergent ($F$)-isocrystal
$$ \mathbf{R}^qf_{\rig*}E \in (F)\textrm{-}\mathrm{Isoc}^\dagger((S,\overline{S})/K).
$$
If $\overline{S}$ is proper, then $\mathbf{R}^qf_{\rig*}E$ moreover only depends on $f:X\rightarrow S$ and $E$.
\end{conjectureu}

Note that if we have a smooth and proper morphism $f:X\rightarrow S$, and extend to a morphism $\bar{f}:\overline{X}\rightarrow \overline{S}$ between compactifications, then the diagram
$$ \xymatrix{ X\ar[r]\ar[d] & \overline{X} \ar[d] \\ S \ar[r] & \overline{S}
}
$$
is Cartesian, and hence Conjecture B1(F) does contain Conjecture B(F) as a special case.

We can also think of overconvergent isocrystals as coherent $j^\dagger\cur{O}$-modules with an overconvergent stratification, and with this viewpoint, a more natural formulation is the following version, due to Shiho.

\begin{conjectureu}[S(F), \cite{Shi08a} Conjecture 5.5] Suppose that $(X,\overline{X})\rightarrow (S,\overline{S})$ is a Cartesian morphism of pairs over $k$, with $\overline{X}\rightarrow \overline{S}$ proper and $X\rightarrow S$ smooth. Let $E$ be an overconvergent ($F$)-isocrystal on $(X,\overline{X})/K$, and $q\geq0$. Then there exists a unique overconvergent ($F$)-isocrystal $\tilde{E}$ on $(S,\overline{S})$ such that for all frames $(T,\overline{T},\frak{T})$ over $S$, with $\frak{T}$ smooth over $\cur{V}$ in a neighbourhood of $T$, the restriction of $\tilde{E}$ to $\mathrm{Strat}^\dagger(T,\overline{T},\frak{T})$ is given by
$$ p_2^* \mathbf{R}^qf'_{\frak{T},\rig*} E|_{(X_T,\overline{X}_{\overline{T}})} \cong  \mathbf{R}^qf'_{\frak{T}\times_\cur{V}\frak{T},\rig*} E|_{(X_T,\overline{X}_{\overline{T}})}\cong p_1^* \mathbf{R}^qf'_{\frak{T},\rig*} E|_{(X_T,\overline{X}_{\overline{T}})}.
$$	
Here $p_i:\frak{T}\times_\cur{V}\frak{T}\rightarrow \frak{T}$ are the projection maps, and $f'$ refers to the natural map of pairs $(X_T,\overline{X}_{\overline{T}}) \rightarrow (T,\overline{T}) $.  If $\overline{S}$ is proper, then $\tilde{E}$ only depends on $f:X\rightarrow S$ and $E$.
\end{conjectureu}

Next, by viewing overconvergent isocrystals as collections of $j^\dagger\cur{O}$-modules on each frame over $(S,\overline{S})$ we arrive at the following, stronger version of Shiho's conjecture.

\begin{conjectureu}[S1(F)] Suppose that $(X,\overline{X})\rightarrow (S,\overline{S})$ is a proper, Cartesian morphism of $k$-pairs, with $X\rightarrow S$ smooth. Let $E$ be an overconvergent ($F$)-isocrystal on $(X,\overline{X})/K$, and $q\geq0$. Then there exists a unique overconvergent ($F$)-isocrystal $\tilde{E}$ on $(S,\overline{S})$ such that for all frames $(T,\overline{T},\frak{T})$ over $(S,\overline{S})$, with $\frak{T}$ smooth over $\cur{V}$ in a neighbourhood of $T$, we have
$$ \tilde{E}_{(T,\overline{T},\frak{T})} \cong \mathbf{R}^qf'_{\frak{T},\rig*} E|_{(X_T,\overline{X}_{\overline{T}})},
$$	
with transition morphisms given by the natural base change morphisms. Here $f'$ is as above. If $\overline{S}$ is proper, then $\tilde{E}$ only depends on $f:X\rightarrow S$ and $E$.
\end{conjectureu}

At least with Frobenius structures, then thanks to full-faithfulness of the restriction from overconvergent to convergent $F$-isocrystals, we get the following (much weaker) form of the conjecture.

\begin{conjectureu}[OF] Let $:(X,\overline{X})\rightarrow (S,\overline{S})$  be a proper, Cartesian morphism of $k$-pairs, with $X\rightarrow S$ smooth, and $E\in F\textrm{-}\mathrm{Isoc}^\dagger((X,\overline{X})/K)$. Then $\mathbf{R}^qf_{\mathrm{conv}*}E$ is overconvergent along $\overline{S}\setminus S$, i.e. is in the essential image of the functor
$$  F\textrm{-}\mathrm{Isoc}^\dagger((S,\overline{S})/K) \rightarrow  F\textrm{-}\mathrm{Isoc}(S/K).
$$
\end{conjectureu}

Finally, by translating into the language of arithmetic $\cur{D}$-modules, we have the following version of Berthelot's conjecture (this has actually been essentially proven by Caro, as we shall see later, but we include it here as a conjecture for the purposes of exposition).

\begin{conjectureu}[C(F)] Suppose that $f: (X,\overline{X})\rightarrow (S,\overline{S})$ is a Cartesian morphism of properly $d$-realisable pairs over $k$, with $\overline{X}\rightarrow \overline{S}$ proper and $X\rightarrow S$ smooth. Then the functor
$$ f_+ :( F\textrm{-})D^b_\mathrm{surcoh}(\cur{D}^\dagger_{(X,\overline{X})/K}) \rightarrow (F\textrm{-})D^b_\mathrm{surcoh}(\cur{D}^\dagger_{(S,\overline{S})/K})
$$
sends $(F\textrm{-})D^b_\mathrm{isoc}(\cur{D}^\dagger_{(X,\overline{X})/K})$ into $(F\textrm{-})D^b_\mathrm{isoc}(\cur{D}^\dagger_{(S,\overline{S})/K})$.
\end{conjectureu}

Thus we have 5 conjectures B, B1, S, S1, C relating to overconvergent isocrystals without Frobenius, and 6 conjectures BF, B1F, SF, S1F, OF, CF relating to those with Frobenius structures. We have the straightforward implications
$$ \textrm{Conjecture S1(F)} \Rightarrow \textrm{Conjecture S(F)}\Rightarrow \textrm{Conjecture B1(F)}\Rightarrow\textrm{Conjecture B(F)}.$$

\begin{lemma} Conjecture \textrm{B1F} $\Rightarrow$ Conjecture \textrm{OF}.
\end{lemma}

\begin{proof}  The question is local on $\overline{S}$, we may therefore assume that we are in the situation of Proposition \ref{occ}, and the lemma immediately follows.
\end{proof}

There are also some implications between the `Frobenius' conjectures and the conjectures without Frobenius, for example we have the obvious observation that Conjecture BF  $\Rightarrow$ Conjecture B, and the base change part of Conjecture S1 means that Conjecture S1 $\Rightarrow$ Conjecture S1F. Also, the existence of the commutative diagram
$$ \xymatrix{  F\textrm{-}D^b_\mathrm{overhol}(\cur{D}^\dagger_{(X,\overline{X})/K}) \ar[r]\ar[d] & D^b_\mathrm{overhol}(\cur{D}^\dagger_{(X,\overline{X})/K}) \ar[d] \\ F\textrm{-}D^b_\mathrm{overhol}(\cur{D}^\dagger_{(S,\overline{S})/K}) \ar[r] & D^b_\mathrm{overhol}(\cur{D}^\dagger_{(S,\overline{S})/K})
}
$$
shows that Conjecture C $\Rightarrow$ Conjecture CF. 

To all of these conjectures we may also append a `base change' statement, which states that the resulting overconvergent $(F\textrm{-})$isocrystals satisfy a suitable base change property via morphisms of varieties $T\rightarrow S$, pairs $(T,\overline{T})\rightarrow (S,\overline{S})$ or triples $(T,\overline{T},\frak{T})\rightarrow (S,\overline{S},\frak{S})$.

For example, in the base of Conjecture B1(F), this states that if $g:(T,\overline{T},\frak{T})\rightarrow (S,\overline{S},\frak{S})$ is a morphism of triples, and $E\in (F\textrm{-})\mathrm{Isoc}^\dagger((X,\overline{X})/K)$ with pullback $E_{(T,\overline{T})}\in (F\textrm{-})\mathrm{Isoc}^\dagger((X_T,\overline{X}_{\overline{T}})/K)$, then in addition to Conjecture B1(F) holding for both the pushforward of $E$ via $f:(X,\overline{X})\rightarrow (S,\overline{S},\frak{S})$ and the pushforward of $E_{(T,\overline{T})}$ via $f_{(T,\overline{T})}:(X_T,\overline{X}_{\overline{T}})\rightarrow (T,\overline{T},\frak{T})$, we have an isomorphism
\[ g^*\mathbf{R}f_{\rig*} E \cong \mathbf{R}f_{(T,\overline{T})\rig*} E_{(T,\overline{T})}\]
of overconvergent ($F\textrm{-}$)isocrystals on $(T,\overline{T})/K$.

Similarly, in the case of Conjecture $C(F)$ this states that if 
\[  \xymatrix{ (X_T,\overline{X}_{\overline{T}}) \ar[r]^{g'}\ar[d]_{f'} & (X,\overline{X})\ar[d]^f  \\ (T,\overline{T})\ar[r]^g & (S,\overline{S})  } \]
is a Cartesian square of properly $d$-realisable pairs, and $E\in ( F\textrm{-})D^b_\mathrm{surhol}(\cur{D}^\dagger_{(X,\overline{X})/K})$ with pullback $g'^! E \in ( F\textrm{-})D^b_\mathrm{surhol}(\cur{D}^\dagger_{(X_T,\overline{X}_{\overline{T}})/K})$, then in addition to Conjecture C(F) holding for both $E$ and $g'^!E$, we have an isomorphism
\[ g^!f_+E \cong f'_+g'^!E \]
in $(F\textrm{-})D^b_\mathrm{surhol}(\cur{D}^\dagger_{(T,\overline{T})/K})$ (for the definition of the extraordinary inverse image functors $g^!$ and $g'^!$ see for example \S4.3 of \cite{Ber02}). We invite the reader to formulate precise `base change' versions of Conjectures B(F), S(F), S1(F) and OF.

We will denote by `c' a conjecture including a base change statement, for example we will refer to Conjecture B1Fc. We therefore have the implications Conjecture B(1)c $\Rightarrow$ Conjecture B(1)Fc, Conjecture Sc $\Rightarrow$ Conjecture SFc and Conjecture S1(F)c $\Leftrightarrow $ Conjecture S1(F).

\section{Previously known results}

In this section, we collect together some previously known cases of the conjectures stated above. We start with the original case noted by Berthelot.

\begin{theorem}[\cite{Ber86}, Th\'{e}or\`{e}me 5] Assume that there exists a morphism $f:\overline{\frak{X}}\rightarrow \overline{\frak{S}}$ of proper formal $\cur{V}$-schemes, and a smooth open formal subscheme $\frak{S}\subset \overline{\frak{S}}$ such that $f:\frak{X}:=f^{-1}\frak{S}\rightarrow \frak{S}$ is smooth and lifts the given morphism $f:X\rightarrow S$. Then Conjecture B(F) holds.
\end{theorem}

In the paper where he introduced Conjecture B1(F), Tsuzuki also proved the following case.

\begin{theorem}[\cite{Tsu03}, Theorem 4.1.1]  In the situation of Conjecture B1(F), suppose that there exists a smooth and proper morphism $(X,\overline{X},\frak{X})\rightarrow (S,\overline{S},\frak{S})$ of smooth $\cur{V}$-frames, such that the square
$$ \xymatrix{  \overline{X} \ar[r] \ar[d] & \frak{X} \ar[d] \\
 \overline{S} \ar[r] & \frak{S}
}
$$
is Cartesian. Then Conjecture B1(F)c holds.
\end{theorem}

Most recently, Caro has proven the following version of Conjecture C(F).

\begin{theorem}[\cite{Car15a}, Th\'{e}or\`{e}mes 4.4.2 and 4.4.3] Conjecture C(F)c holds.\label{cfc}
\end{theorem}

It is also worth mentioning here that Shiho in \cite{Shi08a} proved a weaker version of Conjecture S(F), under certain assumptions on $f$ and $E$, whose statement is somewhat technical and which we will therefore not recall here. There is also a variant on Conjecture C(F) that Caro proved in \cite{Car15a}, which slightly relaxes the condition on $(X,\overline{X})$ and $(S,\overline{S})$ of of begin properly $d$-realisable, but depends on choices of embeddings into formal $\cur{V}$-schemes. 

There are a few more special cases of these conjectures which have been proven by Matsuda-Trihan and by Etesse.

\begin{theorem}[\cite{MT04}, Corollaire 3 and \cite{Ete02}, Th\'{e}or\`{e}me 7] In the situation of Conjecture OF, assume that $S$ is smooth, $\overline{S}$ is proper, $E=\cur{O}_{(X,\overline{X})/K}^\dagger$ is the constant isocrystal, and that the ramification index of $\cur{V}$ is $\leq p-1$. Then Conjecture OF holds in either of the following 2 cases:
\begin{enumerate}\item $S$ is an affine curve.
\item $X$ is an abelian scheme over $S$.
\end{enumerate}
\end{theorem}

\begin{theorem}[\cite{Ete12}, Th\'{e}or\`{e}me 3.4.8.2] In the situation of Conjecture B1(F), assume that $S$ is smooth, and that $(S,\overline{S},\frak{S})$ is a Monsky-Washnitzer frame, with $\cur{S}$ the associated lift of $S$ to a smooth scheme over $\cur{V}$. Assume that $X\rightarrow S$ lifts to a flat morphism $\cur{X}\rightarrow \cur{S}$. Then Conjecture B1(F) holds.
\end{theorem}

\begin{remark} If $S$ is smooth, and $f:X\rightarrow S$ is any smooth and proper morphism which locally lifts to a flat morphism $\cur{X}\rightarrow \cur{S}$ of $\cur{V}$-schemes (for example, if $X$ is a complete intersection in some projective space over $S$), then this is enough to guarantee the existence of unique higher direct image isocrystals $\mathbf{R}^qf_{\rig*}E\in (F\textrm{-})\mathrm{Isoc}^\dagger(X/K)$ 
\end{remark}

\begin{theorem} If $S=\overline{S}$ then Conjectures BF, B1F, SF and S1F are true.
\end{theorem}

\begin{proof} Follows more or less immediately from Proposition \ref{occ}.
\end{proof}

\section{Main results}

We start with the following lemma.

\begin{lemma} \label{compconv} Let $\frak{S}$ be a smooth affine formal $\cur{V}$-scheme with special fibre $S$, and let $f:X\rightarrow S$ be a smooth and projective morphism of $k$-varieties, of constant relative dimension $d$. Then for any convergent isocrystal $E\in \mathrm{Isoc}(X/K)$, with associated arithmetic $\cur{D}$-module $\tilde{E}\in D^b_{\mathrm{isoc}}(\cur{D}_{X/K})$, there is a natural isomorphism
$$ \mathbf{R}^qf_{\frak{S},\mathrm{conv}*}E \cong \cur{H}^{q-d}(f_+\tilde{E})
$$
of $\cur{O}_{\frak{S},\Q}$-modules with integrable connection. This is moreover compatible with base change $\frak{T}\rightarrow \frak{S}$ when $\frak{T}$ is also smooth and affine.
\end{lemma}

\begin{remark}
\begin{enumerate}\item The hypotheses imply that $(S,S)$ and $(X,X)$ are both properly $d$-realisable pairs, hence we do indeed have a push-forward functor
$$ f_+: D^b_\mathrm{surcoh}(\cur{D}_{X/K})\rightarrow D^b_\mathrm{surcoh}(\cur{D}_{S/K}).
$$
\item We have used Lemma \ref{annoying} to identify $D^b_\mathrm{surcoh}(\cur{D}_{S/K})$ with $D^b_\mathrm{surcoh}(\cur{D}^\dagger_{\frak{S},\Q})$, we may therefore consider $\cur{H}^{q-d}(f_+\tilde{E})$ as an $\cur{O}_{\frak{S},\Q}$-module with integrable connection.
\end{enumerate} 
\end{remark}

\begin{proof} Choose closed embeddings $\frak{S}\hookrightarrow \widehat{\A}^n_\cur{V}$ and $X\rightarrow \P^m_S$. Then we may identify $\tilde{E}$ with a certain $\cur{D}$-module on $\frak{P}:=\widehat{\P}^m_{\widehat{\A}^n_\cur{V}}$, supported on $X$, and the push-forward $\tilde{E}$ is given in terms of the functor
$$ g_+: D^b_\mathrm{coh}(\cur{D}^\dagger_{\frak{P},\Q}) \rightarrow D^b(\cur{D}^\dagger_{\widehat{\A}^n_{\cur{V}},\Q})
$$
where $g:\frak{P}\rightarrow \widehat{\A}^n_\cur{V}$ is the projection. Now, the formation of $g_+\tilde{E}\in D^b(\cur{D}^\dagger_{\widehat{\A}^n_{\cur{V}},\Q})$ is local on $\frak{P}$, and the formation of $\mathbf{R}^qf_{\frak{S},\mathrm{conv}*}E$ is local on $X$, therefore we may replace $\frak{P}$ by an open formal subscheme, and $X$ by its intersection with this subscheme, such that $X\hookrightarrow \frak{P}$ lifts to a closed immersion $\frak{X}\hookrightarrow \frak{P}$ of smooth formal $\cur{V}$-schemes. 

But now by the compatibility of the Berthelot-Kashiwara equivalence with push-forwards (which is nothing more than (4.3.6.2) of \cite{Ber02}), we may replace the morphism $g:\frak{P}\rightarrow \widehat{\A}^n_\cur{V}$ by the induced morphism $g:\frak{X}\rightarrow \frak{S}$, and using the concrete description of $\mathrm{sp}_{(X,X),+}$ on page \pageref{sp1}, the claim follows immediately from Proposition \ref{derham} and (4.3.6) of \cite{Ber02}.
\end{proof}

\begin{remark} Note that compatibility with base change means that when $\frak{S}$ is equipped with a lift of the absolute Frobenius on $S$, we can promote the above isomorphism to an isomorphism 
$$ \mathbf{R}^qf_{\frak{S},\mathrm{conv}*}E \cong \cur{H}^{q-d}(f_+\tilde{E})(d)
$$
of (realisations of) convergent $F$-isocrystals.
\end{remark}

Hence we get Conjecture OF in the projective case as follows.

\begin{corollary} \label{ofp} In the situation of Conjecture OF assume that $X\rightarrow S$ is projective. Then Conjecture OF holds.
\end{corollary}

\begin{proof} We may assume that $X\rightarrow S$ has constant relative dimension. Thanks to Chow's lemma, we may blow up $\overline{X}$ outside of $X$ to obtain a projective morphism $\overline{X}\rightarrow \overline{S}$, this does not change the category $\mathrm{Isoc}^\dagger((X,\overline{X})/K)$ and we may therefore assume that $\overline{X}\rightarrow \overline{S}$ is projective. By choosing a convenient alteration $\overline{S}'\rightarrow \overline{S}$ and using Th\'{e}or\`{e}me 2.1.3 of \cite{Car11}, together with base change for $\mathbf{R}^qf_{\mathrm{conv}*}E$, we may assume that $\overline{S}$ is smooth. The question is also local on $\overline{S}$, which we may therefore assume to have a smooth affine lift $\overline{\frak{S}}$ over $\cur{V}$, and that $\overline{\frak{S}}$ is equipped with a lift of the absolute Frobenius of $\overline{S}$. Let $\frak{S}$ denote the open subscheme of $\overline{\frak{S}}$ corresponding to $S$. 

Then question is also local on $S$, hence (after further localising on $\overline{\frak{S}}$) we may assume that there exists a locally closed immersion $\overline{\frak{S}}\rightarrow \widehat{\P}^N_\cur{V}$ and a divisor $D\subset \P^N_k$ such that $S=\overline{S}\setminus D$. Hence $(S,\overline{S})$ is properly $d$-realisable. Since $\overline{X}\rightarrow \overline{S}$ is projective, and 
$$ \xymatrix{ X \ar[r] \ar[d] & \overline{X} \ar[d] \\ S\ar[r] & \overline{S} }
$$
is Cartesian, it follows that the pair $(X,\overline{X})$ is also properly $d$-realisable. Hence for any $\tilde{E}\in F\textrm{-}D^b_\mathrm{isoc}(\cur{D}^\dagger_{(X,\overline{X})/K})$, we have $f_+\tilde{E}\in F\textrm{-}D^b_\mathrm{isoc}(\cur{D}^\dagger_{(S,\overline{S})/K})$. Therefore, using Lemma \ref{compconv} and the proceeding remark, together with Proposition \ref{pfconv}, the realisation $\mathbf{R}^qf_{\frak{S},\mathrm{conv}*}E$ of the convergent $F$-isocrystal $\mathbf{R}^qf_{\mathrm{conv}*}E$ comes from an object in $F\textrm{-}D^b_\mathrm{isoc}(\cur{D}^\dagger_{(S,\overline{S})/K})$, and is thus overconvergent.
\end{proof}

We now turn to a version of Conjecture B1(F).

\begin{lemma} Let $(S,\overline{S},\frak{S})$ be a $\cur{V}$-frame such that $\frak{S}$ is smooth and affine, $\overline{S}=\frak{S}_0$, and $S=\overline{S}\setminus \tilde{D}$ for some divisor $\tilde{D}$ inside some projective embedding $\frak{S}\hookrightarrow \P^N_\cur{V}$. Let $(X,\overline{X})\rightarrow (S,\overline{S})$ be a Cartesian morphism of pairs, with $\overline{X}\rightarrow \overline{S}$ smooth and projective. Let $E\in \mathrm{Isoc}^\dagger((X,\overline{X})/K)$ with associated arithmetic $\cur{D}$-module $\tilde{E}\in D^b_\mathrm{isoc}(\cur{D}^\dagger_{(X,\overline{X})/K})$. Then there is a canonical isomorphism
$$ \mathbf{R}^qf_{\frak{S},\rig*}E \cong \mathrm{sp}^*\cur{H}^{q-d}(f_+\tilde{E})
$$
of $j_S^\dagger\cur{O}_{\frak{S}_K}$-modules with integrable connection.
\end{lemma}

\begin{remark} \begin{enumerate} \item Note again that by Lemma \ref{annoying}, we may view $\cur{H}^{q-d}(f_+\tilde{E})$ as a coherent $\cur{D}^\dagger_{\frak{S},\Q}(^\dagger D)$-module (where $D=\tilde{D}\cap \frak{S}$), and hence as an $\cur{O}_{\frak{S}}(^\dagger D)_\Q$-module with integrable connection. Using the morphism of ringed spaces
$$ \mathrm{sp}: (\frak{S}_K,j_S^\dagger\cur{O}_{\frak{S}_K}) \rightarrow (\frak{S},\cur{O}_\frak{S}(^\dagger D)_\Q)
$$
we may therefore view $\mathrm{sp}^*\cur{H}^{q-d}(f_+\tilde{E})$ as a $j^\dagger_S\cur{O}_{\frak{S}_K}$-module with integrable connection, and the statement of the lemma makes sense.
\item Note that the hypotheses imply that both pairs $(S,\overline{S})$ and $(X,\overline{X})$ are properly $d$-realisable, and hence we are in the situation where Theorem \ref{cfc} holds.
\end{enumerate}
\end{remark}

\begin{proof} Exactly as in the proof of Lemma \ref{compconv}, we may embed $\overline{X}$ into a smooth and proper formal $\frak{S}$ scheme $\frak{P}$, and then localise on $\frak{P}$ to assume that we have a smooth morphism $f:\frak{X}\rightarrow \frak{S}$ lifting $f:\overline{X}\rightarrow \overline{S}$ (although $\overline{X}$ will no longer be proper). Let $D=\tilde{D}\cap \frak{S}$ and $T=f^{-1}D$. We may thus, using the explicit description of $\mathrm{sp}_{(X,\overline{X}),+}$ on page \pageref{sp2}, make the identifications
$$ \tilde{E}\cong \mathrm{sp}_*E, \;\; E \cong \mathrm{sp}^*\tilde{E}
$$
where $\mathrm{sp}:\frak{X}_K\rightarrow \frak{X}$ is the specialisation map. Here again $\mathrm{sp}^*$ refers to module pullback via the morphism
$$
(\frak{X}_K,j_X^\dagger\cur{O}_{\frak{X}_K}) \rightarrow (\frak{X},\cur{O}_{\frak{X}}(^\dagger T)_\Q)
$$
of ringed spaces. Hence using (4.3.6) of \cite{Ber02} again we get a canonical isomorphism
$$ \cur{H}^{q-d}(f_+\tilde{E}) \cong \mathbf{R}^qf_*( \tilde{E} \otimes \Omega^*_{\frak{X}/\frak{S}})
$$
of $\cur{O}_{\frak{S}}(^\dagger D)_\Q$-modules with integrable connection. By using the spectral sequence for a complex, and the identification
$$ \mathrm{sp}^*(\tilde{E}\otimes_{\cur{O}_{\frak{X},\Q}} \Omega^*_{\frak{X}/\frak{S}}) \cong E \otimes_{\cur{O}_{\frak{X}_K}} \Omega^*_{\frak{X}_K/\frak{S}_K},
$$
it therefore suffices to show that for any coherent $\cur{O}_{\frak{S}}(^\dagger D)_\Q$-module $\tilde{E}$, the base change morphism
$$ \mathrm{sp}^*\mathbf{R}^qf_*\tilde{E} \rightarrow \mathbf{R}^qf_{K*} \mathrm{sp}^*\tilde{E} 
$$
is an isomorphism. Actually, since overconvergent isocrystals extend to some strict neighbourhood of $]X[_\frak{X}$, we may by Proposition 4.4.5 of \cite{Ber96a} assume that there exists some $r$ such that $\tilde{E}$ comes from a coherent $\cur{B}_\frak{X}(T,r_0)_\Q$-module for some $r_0\geq0$, i.e. we have
$$ \tilde{E}\cong (\mathrm{colim}_{r\geq r_0} E_r)_\Q
$$
for coherent $\cur{B}_\frak{X}(T,r)$-modules $E_r$. (These $\cur{B}_\frak{X}(T,r)$ are essentially formal models for the ring of functions on a certain cofinal system of neighbourhoods of $]X[_\frak{X}$ inside $\frak{X}_K$, for more details see \S4 of \cite{Ber96a}.) Since $\mathrm{sp}^*,\mathbf{R}f_*$ and $\mathbf{R}f_{K*}$ commute with filtered direct limits ($f$ and $f_K$ are both quasi-compact), we can therefore reduce to the case of a coherent $\cur{B}_\frak{X}(T,r)$-module $E$. But now we may replace $\frak{X}$ by the relative spectrum $\mathrm{Spf}(\cur{B}_\frak{X}(T,r))$, it therefore suffices to treat the case of a coherent $\cur{O}_\frak{X}$-module. By further localising on $\frak{X}$ we may assume it to be affine, whence it suffices to treat the case $q=0$. This then follows by direct calculation.
\end{proof}

Of course, as with Lemma \ref{compconv}, there exists a version with Frobenius, and we easily arrive at the following.

\begin{corollary} \label{bp} In the situation of Conjecture B1(F), assume that $\frak{S}$ is smooth, $\overline{S}=\frak{S}_0$ and that the induced morphism $\overline{X}\rightarrow \overline{S}$ is smooth and projective. Then Conjecture B1(F) holds.
\end{corollary}

\begin{proof} Entirely similar to the proof of Corollary \ref{ofp}.
\end{proof}

\begin{remark} Note that by using Th\'{e}or\`{e}me 4.4.2 of \cite{Car15a} and Proposition 4.1.8 of \cite{Car09a} we also get a base change statement for morphisms $(T,\overline{T},\frak{T}) \rightarrow (S,\overline{S},\frak{S})$ where  $(T,\overline{T},\frak{T})$ also satisfies the hypotheses of the corollary.
\end{remark}

\section*{Acknowledgements}

The author would like to thank Ambrus P\'{a}l, whose interest in the current status of Berthelot's conjecture (in particular, the requirement of a result along the lines of Corollary \ref{ofp} to be used in \cite{Pal15b}) lead to the writing of this article. The author was supported by an HIMR fellowship.

\bibliographystyle{mysty}
\bibliography{/Users/Chris/Dropbox/LaTeX/lib.bib}
  
\end{document}